\documentclass[12pt,leqno,oneside]{amsart}
\usepackage{mathrsfs,dsfont}
\usepackage{amsmath,amstext,amsthm,amssymb,bbm}
\usepackage{charter}
\usepackage{typearea}
\usepackage{hyperref}
\usepackage{color}

\usepackage{float}
\usepackage{tikz}
\usetikzlibrary{patterns}
\usepackage{graphicx}
\usepackage{hyperref}

\usepackage[width=6.4in,height=8.5in]
{geometry}

\newtheorem{Theorem}{Theorem}[section]
\newtheorem{Proposition}[Theorem]{Proposition}
\newtheorem{Lemma}[Theorem]{Lemma}

\theoremstyle{definition}

\newtheorem{Example}[Theorem]{Example}

\numberwithin{equation}{section}

\DeclareMathOperator{\Span}{Span}

\newcommand{\comment}[1]{}

\begin{document}

\title{Lelong numbers of bidegree $(1,1)$  currents on multiprojective spaces}

\author{Dan Coman}
\thanks{D.\ Coman is partially supported by the NSF Grant DMS-1700011}
\address{Department of Mathematics, Syracuse University, Syracuse, NY 13244-1150, USA}
\email{dcoman@syr.edu}

\author{James Heffers}
\address{Department of Mathematics, University of Michigan, Ann Arbor, MI 48109-1043, USA}
\email{heffers@umich.edu}


\date{January 29, 2019}

\begin{abstract} 
Let $T$ be a positive closed current of bidegree $(1,1)$ on a multiprojective space $X={\mathbb P}^{n_1}\times\ldots\times{\mathbb P}^{n_k}$. For certain values of $\alpha$, which depend on the cohomology class of $T$, we show that the set of points of $X$ where the Lelong numbers of $T$ exceed $\alpha$ have certain geometric properties. We also describe the currents $T$ that have the largest possible Lelong number in a given cohomology class, and the set of points where this number is assumed.
\end{abstract}

\maketitle

\tableofcontents

\section{Introduction}\label{S:intro}
Let $M$ be a complex manifold of dimension $n$ and $T$ be a positive closed current of bidimension $(p,p)$ (or bidegree $(n-p,n-p)$) on $M$. Consider the upper level sets  
\[E_\alpha(T):=\{x\in M:\,\nu(T,x)\geq\alpha\}\,,\,\;E^+_\alpha(T):=\{x\in M:\,\nu(T,x)>\alpha\}\,,\]
where $\nu(T,x)$ is the Lelong number of $T$ at $x\in M$ and $\alpha\geq0$ (see \cite{D93,Ho} for the definition and properties of Lelong numbers). A fundamental theorem of Siu \cite{Siu74} states that for $\alpha>0$, $E_\alpha(T)$ is an analytic subvariety of $M$ of dimension at most $p$. It follows that $E^+_0(T)$ is an at most countable union of analytic subvarieties of $M$ of dimension at most $p$. 

In the case of the projective space $M={\mathbb P}^n$, explicit geometric descriptions of the sets $E^+_\alpha(T)$ are obtained in \cite{C06} and \cite{CT15}. Further results in this direction are given in \cite{He17,He19}. We also note that the case of currents of bidimension $(1,1)$ on multiprojective spaces is studied in \cite[Section 4]{CT15}.

Our goal here is to study the geometric properties of the sets $E^+_\alpha(T)$ for positive closed currents $T$ of bidegree $(1,1)$ on multiprojective spaces. Throughout the paper we let  
\begin{equation}\label{e:X}
X:={\mathbb P}^{n_1}\times\ldots\times{\mathbb P}^{n_k}={\mathbb P}^{n_1}_{[z^1]}\times\ldots\times{\mathbb P}^{n_k}_{[z^k]}\,,\,\;n:=n_1+\ldots+n_k\,,
\end{equation}
where $z^j=(z^j_0,\ldots,z^j_{n_j})\in{\mathbb C}^{n_j+1}$ and $[z^j]=[z^j_0:\ldots:z^j_{n_j}]$ denote the homogeneous coordinates on $\mathbb P^{n_j}$. Let 
\begin{equation}\label{e:proj}
\Pi_j:{\mathbb C}^{n_j+1}\setminus\{0\}\to\mathbb P^{n_j},\;\Pi_j(z^j)=[z^j]\,,\quad\pi_j:X\to\mathbb P^{n_j},
\end{equation}
be the canonical projection, and respectively the projection onto the $j$-th factor. Set 
\[\omega_j=\pi^\star_j\omega_{FS}\,,\,\;1\leq j\leq k\,,\]
where $\omega_{FS}$ denotes the Fubini-Study K\"ahler form on a projective space $\mathbb P^{n_j}$. The Dolbeault cohomology group $H^{1,1}(X,\mathbb R)$ is generated by the forms $\omega_1,\ldots,\omega_k$. 

We let 
\begin{equation}\label{e:om}
(a_1,\ldots,a_k)\in(0,+\infty)^k\,,\,\;a:=a_1+\ldots+a_k\,,\,\;\omega=\omega_{a_1,\ldots,a_k}=a_1\omega_1+\ldots+a_k\omega_k\,,
\end{equation}
and we denote by 
\begin{equation}\label{e:T}
\mathcal T=\mathcal T_{a_1,\ldots,a_k}(X)
\end{equation}
the space of positive closed currents $T$ of bidegree $(1,1)$ on $X$ in the cohomology class of $\omega$ (i.e. $T\sim\omega$).

\smallskip

In the above setting, our first result gives a description of the currents in $\mathcal T$ with the largest possible Lelong number. It is the analogue of \cite[Proposition 2.3]{CG09} to the case of multiprojective spaces (see also Proposition \ref{P:proj} in the following section). The case of bidegree $(1,1)$ currents on ${\mathbb P}^1\times{\mathbb P}^1$ was treated in \cite[Proposition 4.1]{CG09}.

\begin{Theorem}\label{T:t1}
If $T\in\mathcal T$ then $\nu(T,x)\leq a$ for all $x\in X$. If $E_a(T)\neq\emptyset$ then there exist proper linear subspaces $L_j\subset{\mathbb P}^{n_j}$ of dimension $\ell_j$, and surjective linear maps $\eta_j:{\mathbb C}^{n_j+1}\to{\mathbb C}^{n_j-\ell_j}$, $1\leq j\leq k$, such that 
\[E_a(T)=L_1\times\ldots\times L_k\,,\,\;L_j=\Pi_j(\ker\eta_j\setminus\{0\})\,,\,\;T=\wp^\star S\,,\]
where $\wp=[\eta_1]\times\ldots\times[\eta_k]:X\dashrightarrow Y:={\mathbb P}^{n_1-\ell_1-1}\times\ldots\times{\mathbb P}^{n_k-\ell_k-1}$, $[\eta_j]:{\mathbb P}^{n_j}\dashrightarrow{\mathbb P}^{n_j-\ell_j-1}$ is the projection induced by $\eta_j$, and $S\in\mathcal T_{a_1,\ldots,a_k}(Y)$ is a current with $E_a(S)=\emptyset$.
\end{Theorem}

Our next result extends \cite[Proposition 2.2]{CG09} to the case of multiprojective spaces (see also \cite[Theorem 1.1]{CT15} for the case of currents of arbitrary bidegree on projective space). 

\begin{Theorem}\label{T:t2}
If $T\in\mathcal T$, $\nu_j=a-\frac{a_j}{n_j+1}\,$ and $\nu_0=\max\{\nu_1,\ldots,\nu_k\}$, then the following hold:

(i) There exist proper linear subspaces $V_j\subset{\mathbb P}^{n_j}$ such that $\pi_j\big(E^+_{\nu_j}(T)\big)\subset V_j$, for $1\leq j\leq k$.

(ii) $E^+_{\nu_0}(T)\subset V_1\times\ldots\times V_k$.
\end{Theorem}

The following result is a version of \cite[Theorem 1.2]{CT15} for currents of bidegree $(1,1)$ on multiprojective spaces. If $A\subset{\mathbb P}^N$ we denote by $\Span A$ the smallest linear subspace of ${\mathbb P}^N$ containing $A$. 

\begin{Theorem}\label{T:t3}
Let $T\in\mathcal T$ and set $\nu_j=a-\frac{a_j}{n_j}\,$, $\nu_0=\max\{\nu_1,\ldots,\nu_k\}$. We have the following:

(i) If $n_j\geq2$ for some $1\leq j\leq k$, then the set $\pi_j\big(E^+_{\nu_j}(T)\big)$ is either contained in a hyperplane $H_j$ of ${\mathbb P}^{n_j}$, or else it is a finite set and $\pi_j\big(E^+_{\nu_j}(T)\big)\setminus L_j=A_j$, for some line $L_j$ and set $A_j$ with $|A_j|=n_j-1$ and $\Span(L_j\cup A_j)={\mathbb P}^{n_j}$.

(ii) If $n_j\geq2$ for all $1\leq j\leq k$, then $E^+_{\nu_0}(T)\subset W_1\times\ldots\times W_k$, where for each $j$, $W_j=H_j$ or $W_j=L_j\cup A_j$. Moreover, if $W_j=L_j\cup A_j$ for all $1\leq j\leq k$, then $E^+_{\nu_0}(T)$ is a finite set.
\end{Theorem}

One can also obtain a version of \cite[Theorem 1.1]{He19} in the case of multiprojective spaces. It strengthens the conclusion of Theorem \ref{T:t3} under the additional assumption on the existence of two points where $T$ has large Lelong number. 

\begin{Theorem}\label{T:t4}
Let $T\in\mathcal T$ and set $\nu_j=a-\frac{a_j}{n_j}\,$, $\nu_0=\max\{\nu_1,\ldots,\nu_k\}$.  Furthermore, let 
\[\beta_j=\beta_j(\alpha_j)=\frac{an_j^2-a_jn_j-\alpha_j}{n_j^2-1}\;,\,\text{ where } \nu_j<\alpha_j\leq a\,,\]
and set $\beta_0 = \max \{\beta_1, \ldots, \beta_k\}$, $\alpha_0 = \max\{\alpha_1, \ldots, \alpha_k\}$.
Let $p=(p_1,\ldots,p_k)\in X$, $q=(q_1,\ldots,q_k)\in X$. Then the following hold:

(i) If $n_j\geq2$, $p_j\neq q_j$, $\nu(T, p)\geq \alpha_j$, $\nu(T,q) \geq \alpha_j$, for some $1\leq j\leq k$, then the set $\pi_j\big(E^+_{\beta_j}(T)\big)$ is either contained in a hyperplane $H_j$ of ${\mathbb P}^{n_j}$, or else $\pi_j\big(E^+_{\beta_j}(T)\big)\setminus L_j=A_j$, for some line $L_j$ and set $A_j$ with $|A_j|=n_j-1$ and $\Span(L_j\cup A_j)={\mathbb P}^{n_j}$.

(ii) If $n_j\geq2$, $p_j\neq q_j$, for all $1\leq j \leq k$, and if $\nu(T, p)\geq \alpha_0$, $\nu(T,q) \geq \alpha_0$, then $E^+_{\beta_0}(T)\subset W_1\times\ldots\times W_k$, where for each $j$, $W_j=H_j$ or $W_j=L_j\cup A_j$. 
\end{Theorem}

It is worth noting that $\beta_j(\nu_j) = \nu_j$, and that $\beta_j$ decreases as $\alpha_j$ increases. 

\medskip

The paper is organized as follows. In Section \ref{S:prelim} we recall results about the structure of positive closed currents of bidegree $(1,1)$ on multiprojective spaces. In Section \ref{S:psh} we develop some of the tools needed for the proof of our results; see Proposition \ref{P:psh} and Theorem \ref{T:psh}, which deal with growth properties of entire plurisubharmonic (psh) functions in certain Lelong classes on ${\mathbb C}^n$ that have a large Lelong number at the origin. Theorems \ref{T:t1}, \ref{T:t2}, \ref{T:t3}, and \ref{T:t4} are proved in Section \ref{S:proofs}. We also give there examples showing that these results are sharp. In Section \ref{S:fr} we consider positive closed currents of bidegree $(1,1)$ on a projective space, and we obtain in Theorem \ref{T:proj} a more precise version of \cite[Proposition 2.2]{CG09}.

\section{Preliminaries}\label{S:prelim}
Positive closed currents of bidegree $(1,1)$ on a projective space ${\mathbb P}^m$ can be described via their logarithmically homogeneous plurisubharmonic (psh) potentials on ${\mathbb C}^{m+1}$ or via psh functions in the Lelong class on ${\mathbb C}^m\hookrightarrow{\mathbb P}^m$ (see \cite{FS95d,Si99,GZ05}). Recall that the Lelong class ${\mathcal L}({\mathbb C}^m)$ is the class of psh functions $u$ on ${\mathbb C}^m$ that satisfy $u(z)\leq\log^+|z|+C_u$ for all $z\in{\mathbb C}^m$, with some constant $C_u$ depending on $u$. A similar description holds in the case of multiprojective spaces and we recall it in this section (see also \cite[Section 2]{FG01}). If $M$ is a complex manifold and $\Omega$ is a smooth real $(1,1)$-form on $M$, an $\Omega$-plurisubharmonic ($\Omega$-psh) function on $M$ is a function $\psi$ which is locally the sum of a psh function and a smooth one, and which verifies $\Omega+dd^c\psi\geq0$ in the sense of currents. Here $d=\partial+\overline\partial$, $d^c=\frac{1}{2\pi i}(\partial-\overline\partial)$. 

Let $X$ be the multiprojective space defined in \eqref{e:X} endowed with the K\"ahler form $\omega$ from \eqref{e:om}. Recall the definition \eqref{e:proj} of the projections $\Pi_j$ and $\pi_j$, and set 
\begin{equation}\label{e:Pi}
\Pi=\Pi_1\times\ldots\times\Pi_k:({\mathbb C}^{n_1+1}\setminus\{0\})\times\ldots\times({\mathbb C}^{n_k+1}\setminus\{0\})\to X\,.
\end{equation}
Consider the standard embeddings
\begin{equation}\label{e:emb}
\begin{split}
&{\mathbb C}^{n_j}\hookrightarrow{\mathbb P}^{n_j},\,\;\zeta^j=(\zeta^j_1,\ldots,\zeta^j_{n_j})\in{\mathbb C}^{n_j}\to[1:\zeta^j]:=[1:\zeta^j_1:\ldots:\zeta^j_{n_j}]\in{\mathbb P}^{n_j},\\
&{\mathbb C}^n={\mathbb C}^{n_1}\times\ldots\times{\mathbb C}^{n_k}\hookrightarrow X\,,\,\;(\zeta^1,\ldots,\zeta^k)\to([1:\zeta^1],\ldots,[1:\zeta^k])\,.
\end{split}
\end{equation}

Let $T\in{\mathcal T}$, where ${\mathcal T}$ is defined in \eqref{e:T}. Then 
 \begin{equation}\label{e:omp}
 T=\omega+dd^c\varphi\,,
 \end{equation}
 where $\varphi=\varphi_T$ is an $\omega$-psh function on $X$, unique up to additive constants. We define the function $U=U_T$ by 
 \begin{equation}\label{e:logp}
 U(z^1,\ldots,z^k)=\sum_{j=1}^ka_j\log|z^j|+\varphi([z^1],\ldots,[z^k])\,.
 \end{equation}
 Then $U$ extends to a psh function on ${\mathbb C}^{n_1+1}\times\ldots\times{\mathbb C}^{n_k+1}$, $\Pi^\star T=dd^c U$, and $U$ satisfies the logarithmic homogeneity condition 
\begin{equation}\label{e:logh}
U(t_1z^1,\ldots,t_kz^k)=\sum_{j=1}^ka_j\log|t_j|+U(z^1,\ldots,z^k),\,\;\forall\,t_j\in{\mathbb C}\setminus\{0\},\;1\leq j\leq k\,.
\end{equation}

Set $\;u=u_T\in PSH({\mathbb C}^n)$, $u(\zeta^1,\ldots,\zeta^k)=U(1,\zeta^1,\ldots,1,\zeta^k)$. Then it is easy to see that $T\mid_{{\mathbb C}^n}=dd^cu$ and that the function $u\in a{\mathcal L}({\mathbb C}^n)$ (where $a$ is defined in \eqref{e:om}) satisfies the special growth condition 
\begin{equation}\label{e:Lp}
u(\zeta^1,\ldots,\zeta^k)\leq\sum_{j=1}^ka_j\log^+|\zeta^j|+C\,,\,\text{ for some constant $C$.}
\end{equation}

Conversely, if $u\in PSH({\mathbb C}^n)$ satisfies \eqref{e:Lp} then the function 
\begin{equation}\label{e:logL}
U(t_1,\zeta^1,\ldots,t_k,\zeta^k)=\sum_{j=1}^ka_j\log|t_j|+u(\zeta^1/t_1,\ldots,\zeta^k/t_k)\,,\,\;t_j\in\mathbb C\setminus\{0\}\,,\,\;\zeta^j\in{\mathbb C}^{n_j},
\end{equation}
extends to a psh function on ${\mathbb C}^{n_1+1}\times\ldots\times{\mathbb C}^{n_k+1}$ which satisfies \eqref{e:logh}. Thus $u$ determines a current in $\mathcal T$.  

\smallskip

We will need the following result which is contained in \cite[Propositions 2.1 and 2.3]{CG09}. It gives a description of positive closed currents of bidegree $(1,1)$ on ${\mathbb P}^m$ with highest Lelong number. Recall that if $T$ is a positive closed current of bidegree $(1,1)$ on ${\mathbb P}^m$ its mass is given by 
\[\|T\|=\int_{{\mathbb P}^m}T\wedge\omega_{FS}^{m-1}.\]
Moreover, if $\|T\|=1$ then $T=\omega_{FS}+dd^c\varphi$ for some $\omega_{FS}$-psh function $\varphi$ on ${\mathbb P}^m$. Let $\Pi:{\mathbb C}^{m+1}\setminus\{0\}\to{\mathbb P}^m$ be the canonical projection, set $z=(z_0,\ldots,z_m)\in{\mathbb C}^{m+1}$, $\Pi(z):=[z]=[z_0:\ldots:z_m]\in{\mathbb P}^m$.

\begin{Proposition}\label{P:proj}
Let $T$ be a positive closed current of bidegree $(1,1)$ on ${\mathbb P}^m$ with $\|T\|=1$. 

(i) We have $\nu(T,x)\leq 1$ for all $x\in{\mathbb P}^m$. 

(ii) If $E_1(T)\neq\emptyset$ then $E_1(T)$ is a proper linear subspace of ${\mathbb P}^m$.

(iii) Let $L\subset E_1(T)$ be a linear subspace of dimension $\ell$ and $\eta:{\mathbb C}^{m+1}\to{\mathbb C}^{m-\ell}$ be a surjective linear map such that $L=\Pi(\ker\eta\setminus\{0\})$. Then 
\[T=\omega_{FS}+dd^c\left(\log\frac{|\eta(z)|}{|z|}+h([\eta(z)])\right),\]
where $h$ is an $\omega_{FS}$-psh function on ${\mathbb P}^{m-\ell-1}$, $[\eta(z)]=[\eta]([z])$, and $[\eta]:{\mathbb P}^m\dashrightarrow{\mathbb P}^{m-\ell-1}$ is the projection induced by $\eta$. 

(iv) If $\dim E_1(T)=\ell$ and $\eta:{\mathbb C}^{m+1}\to{\mathbb C}^{m-\ell}$ is a surjective linear map such that $E_1(T)=\Pi(\ker\eta\setminus\{0\})$, then $T=[\eta]^\star S$, where $S$ is a positive closed current of bidegree $(1,1)$ on ${\mathbb P}^{m-\ell-1}$ with  $\|S\|=1$ and $E_1(S)=\emptyset$.
\end{Proposition}

\begin{proof} We give a sketch of the proof for the convenience of the reader. Assertion $(i)$ is contained in \cite[Proposition 2.1]{CG09}, while $(ii)$, $(iv)$ in \cite[Proposition 2.3]{CG09}. For $(iii)$, write $T=\omega_{FS}+dd^c\varphi$, where $\varphi$ is an $\omega_{FS}$-psh function on ${\mathbb P}^{m}$. Since $\eta$ is linear and surjective, there exists a linear isomorphism $A:{\mathbb C}^{m+1}\to{\mathbb C}^{m+1}$ such that $\eta\circ A(t)=(t_{\ell+1},\ldots,t_m)$, where $t=(t_0,\ldots,t_m)$. Hence $A(V)=\ker\eta$, where $V=\{t_{\ell+1}=\ldots=t_m=0\}$. Let $[A(t)]=[A]([t])$, where $[A]:{\mathbb P}^m\dashrightarrow{\mathbb P}^m$ is the automorphism of ${\mathbb P}^m$ induced by $A$. If 
\[S=[A]^\star T=\omega_{FS}+dd^c\left(\log\frac{|A(t)|}{|t|}+\varphi([A(t)])\right),\]
then $\Pi(V\setminus\{0\})\subset E_1(S)$. By Proposition 2.3 in \cite{CG09} and its proof we infer that 
\[\log\frac{|A(t)|}{|t|}+\varphi([A(t)])=\log\frac{|t_{\ell+1}|^2+\ldots+|t_m|^2}{|t_0|^2+\ldots+|t_m|^2}+h([t_{\ell+1}:\ldots:t_m])\,,\]
for some $\omega_{FS}$-psh function $h$ on ${\mathbb P}^{m-\ell-1}$. Thus 
\[\begin{split}
T&=[A^{-1}]^\star S=[A^{-1}]^\star\omega_{FS}+dd^c\left(\log\frac{|\eta(z)|}{|A^{-1}(z)|}+h([\eta(z)])\right)\\
&=\omega_{FS}+dd^c\left(\log\frac{|\eta(z)|}{|z|}+h([\eta(z)])\right).
\end{split}\]
\end{proof}

\section{Plurisubharmonic functions in special Lelong classes}\label{S:psh}
We study here entire psh functions that satisfy certain growth conditions. The results will be used in the proofs of our main theorems. The first proposition deals with the case of psh functions in certain Lelong classes, with the largest possible Lelong number at the origin.

\begin{Proposition}\label{P:psh}
Let $u$ be a psh function on ${\mathbb C}^n={\mathbb C}^{n_1}\times\ldots\times{\mathbb C}^{n_k}$ verifying \eqref{e:Lp}. Then $\nu(u,0)\leq a=a_1+\ldots+a_k$. Moreover, if $\nu(u,0)=a$ then 
\[u(\zeta^1,\ldots,\zeta^k)\leq\sum_{j=1}^ka_j\log|\zeta^j|+C \;\text{ on ${\mathbb C}^n$}\,,\]
with the same constant $C$ as in \eqref{e:Lp}. In particular, $\nu(u,x)\geq a_j$ for all $x\in\{\zeta^j=0\}$. 
\end{Proposition}

\begin{proof}
By \eqref{e:Lp} we have that $u\in a{\mathcal L}({\mathbb C}^n)$. Let $R$ be the trivial extension to ${\mathbb P}^n$ of the current $dd^cu$. Then $\|R\|\leq a$, so $\nu(u,0)\leq a$ by Proposition \ref{P:proj}. If $\nu(u,0)=a$ then \cite[Proposition 2.1]{CG09} implies that 
\[u(\zeta^1,\ldots,\zeta^k)=\frac{a}{2}\,\log(|\zeta^1|^2+\ldots+|\zeta^k|^2)+h([\zeta^1:\ldots:\zeta^k])\,,\]
for some $a\omega_{FS}$-psh function $h$ on ${\mathbb P}^{n-1}$. If $\zeta^j\neq0$ for all $1\leq j\leq k$ and $t\in\mathbb C$ is such that $|t\zeta^j|\geq1$ for all $1\leq j\leq k$, we obtain using \eqref{e:Lp} that
\begin{equation*}
\begin{split}
h([\zeta^1:\ldots:\zeta^k])
&=h([t\zeta^1:\ldots:t\zeta^k])=u(t\zeta^1,\ldots,t\zeta^k)-\frac{a}{2}\,\log(|t\zeta^1|^2+\ldots+|t\zeta^k|^2)\\
&\leq\sum_{j=1}^ka_j\log|t\zeta^j|+C-\frac{a}{2}\,\log(|t\zeta^1|^2+\ldots+|t\zeta^k|^2)\\
&=\sum_{j=1}^ka_j\log|\zeta^j|+C-\frac{a}{2}\,\log(|\zeta^1|^2+\ldots+|\zeta^k|^2)\,.
\end{split}
\end{equation*}
This yields the conclusion.
\end{proof}

Our next result is a refinement of Proposition \ref{P:psh} and it deals with the case of psh functions $u$ that have a sufficiently large Lelong number at the origin. 

\begin{Theorem}\label{T:psh}
Let $u$ be a psh function on ${\mathbb C}^n={\mathbb C}^{n_1}\times\ldots\times{\mathbb C}^{n_k}$ verifying \eqref{e:Lp} and set $\nu:=\nu(u,0)$. 

(i) If $\nu\geq a-a_j$, for some $1\leq j\leq k$, then 
\[u(\zeta^1,\ldots,\zeta^k)\leq(\nu+a_j-a)\log|\zeta^j|+ (a-\nu)\log^+|\zeta^j|+\sum_{\ell=1,\ell\neq j}^ka_\ell\log^+|\zeta^\ell|+C \;\text{ on ${\mathbb C}^n$}\,,\]
with the same constant $C$ as in \eqref{e:Lp}.

(ii) If $\nu\geq\max\{a-a_1,\ldots,a-a_k\}$ then 
\[u(\zeta^1,\ldots,\zeta^k)\leq\sum_{j=1}^k(\nu+a_j-a)\log|\zeta^j|+(a-\nu)\sum_{j=1}^k\log^+|\zeta^j|+C \;\text{ on ${\mathbb C}^n$}\,,\]
with the same constant $C$ as in \eqref{e:Lp}
\end{Theorem}

\begin{proof}
By Proposition \ref{P:psh} we have that $\nu\leq a$. We divide the proof in three steps.

{\em Step 1.} We assume here that $n_1=\ldots=n_k=1$ and $v$ is a psh function on ${\mathbb C}^k$ such that  
\begin{equation}\label{e:L1}
v(t_1,\ldots,t_k)\leq\sum_{j=1}^ka_j\log^+|t_j|+C\,,\;\forall\,(t_1,\ldots,t_k)\in{\mathbb C}^k\,,
\end{equation}
for some constant $C$, and $\nu:=\nu(v,0)\geq a-a_j$ for some $j$. We show that $(i)$ holds for $v$. 

Assume without loss of generality that $j=1$ and let $\alpha_1$ denote the generic Lelong number of $v$ along $\{t_1=0\}$. Then $\alpha_1\leq\nu$. Moreover, if $t':=(t_2,\ldots,t_k)$ is such that the function $v(\cdot,t')\not\equiv-\infty$, then by \eqref{e:L1}, $v(\cdot,t')\in a_1{\mathcal L}(\mathbb C)$.  Using Proposition \ref{P:proj} we infer that  
\[\nu(v,(0,t'))\leq\nu(v(\cdot,t'),0)\leq a_1\,,\,\text{ so }\,\alpha_1\leq a_1\,.\]

By Siu's decomposition theorem \cite{Siu74}, $dd^cv=\alpha_1[t_1=0]+R$, where $[t_1=0]$ denotes the current of integration along the hyperplane $\{t_1=0\}$ and $R$ is a positive closed current of bidegree $(1,1)$ on ${\mathbb C}^k$, with generic Lelong number $0$ along $\{t_1=0\}$. Hence $w:=v-\alpha_1\log|t_1|$ extends to a psh function on ${\mathbb C}^k$, which satisfies $dd^cw=R$ and 
\begin{equation}\label{e:L2}
w(t_1,\ldots,t_k)\leq(a_1-\alpha_1)\log^+|t_1|+\sum_{j=2}^ka_j\log^+|t_j|+C\,,\,\text{ $\forall\,(t_1,\ldots,t_k)\in{\mathbb C}^k$,}
\end{equation}
where $C$ is the constant from \eqref{e:L1}. Indeed, this clearly holds if $|t_1|\geq1$. Applying the maximum principle for $w(\cdot,t')$, with $t'$ fixed, then shows that \eqref{e:L2} holds everywhere. 

We now estimate $\alpha_1$. Consider the current $S$ on $({\mathbb P}^1)^k$ determined by $w$, so $S\mid_{{\mathbb C}^k}=R$ (see \eqref{e:Lp} and \eqref{e:logL}). By \eqref{e:L2} we have
\[S\sim(a_1-\alpha_1)\omega_{t_1}+\sum_{j=2}^ka_j\omega_{t_j}\,,\]
where $\omega_{t_j}=\pi_j^\star\omega_{FS}$ and $\pi_j$ is the projection onto the $j$-th factor. By Demailly's regularization theorem \cite[Proposition 3.7]{D92}, there exists, for every $\varepsilon>0$, a positive closed current $S_\varepsilon$ of bidegree $(1,1)$ on $({\mathbb P}^1)^k$, with analytic singularities and such that 
\[S_\varepsilon\sim(a_1-\alpha_1+\varepsilon)\omega_{t_1}+\sum_{j=2}^k(a_j+\varepsilon)\omega_{t_j}\,,\,\;\nu(S,x)-\varepsilon\leq\nu(S_\varepsilon,x)\leq\nu(S,x),\;\forall\,x\in({\mathbb P}^1)^k\,.\]
Thus $S_\varepsilon$ is smooth near all points where $\nu(S,x)=0$, and in particular near the generic point of $\{t_1=0\}$. Let $w_\varepsilon$ be the psh potential of $S_\varepsilon$ on ${\mathbb C}^k$ defined in \eqref{e:Lp}. Then 
\[w_\varepsilon(t_1,\ldots,t_k)\leq(a_1-\alpha_1+\varepsilon)\log^+|t_1|+\sum_{j=2}^k(a_j+\varepsilon)\log^+|t_j|+C_\varepsilon\,,\,\text{ $\forall\,(t_1,\ldots,t_k)\in{\mathbb C}^k$,}\]
for some constant $C_\varepsilon$. Moreover, $\nu(w_\varepsilon,0)\geq\nu(w,0)-\varepsilon=\nu-\alpha_1-\varepsilon$, and $w_\varepsilon$ is smooth near the generic point of $\{t_1=0\}$. Thus $w_\varepsilon(0,\cdot)$ is psh on ${\mathbb C}^{k-1}$ and satisfies
\[w_\varepsilon(0,t')\leq\left(\sum_{j=2}^ka_j+(k-1)\varepsilon\right)\log^+|t'|+C_\varepsilon\,.\]
We infer by Proposition \ref{P:proj} that 
$\nu-\alpha_1-\varepsilon\leq\nu(w_\varepsilon,0)\leq\nu(w_\varepsilon(0,\cdot),0)\leq a-a_1+(k-1)\varepsilon$.
Letting $\varepsilon\searrow0$ yields that $\alpha_1\geq\nu+a_1-a$. 

By \eqref{e:L2},
\[v(t_1,\ldots,t_k)\leq\alpha_1\log|t_1|+(a_1-\alpha_1)\log^+|t_1|+\sum_{j=2}^ka_j\log^+|t_j|+C\,.\]
Since $\alpha_1\geq\nu+a_1-a$ it follows that 
\[v(t_1,\ldots,t_k)\leq(\nu+a_1-a)\log|t_1|+(a-\nu)\log^+|t_1|+\sum_{j=2}^ka_j\log^+|t_j|+C\,,\]
for all $(t_1,\ldots,t_k)\in{\mathbb C}^k$. This concludes Step 1. 

\medskip

{\em Step 2.} We show here that assertion $(ii)$ of Theorem \ref{T:psh} holds for functions $v$ that verify \eqref{e:L1}, if $\nu:=\nu(v,0)\geq\max\{a-a_1,\ldots,a-a_k\}$. Let $\alpha_j$ denote the generic Lelong number of $v$ along $\{t_j=0\}$. Then $\alpha_j\leq\nu$ and by Step 1, $\nu+a_j-a\leq\alpha_j\leq a_j$ for all $1\leq j\leq k$. By Siu's decomposition theorem \cite{Siu74}, the function 
\[w(t_1,\ldots,t_k):=v(t_1,\ldots,t_k)-\sum_{j=1}^k\alpha_j\log|t_j|\] 
extends to a psh function on ${\mathbb C}^k$, which satisfies 
\[w(t_1,\ldots,t_k)\leq\sum_{j=1}^k(a_j-\alpha_j)\log^+|t_j|+C\,,\,\text{ $\forall\,(t_1,\ldots,t_k)\in{\mathbb C}^k$,}\]
where $C$ is the constant from \eqref{e:L1}. Hence 
\[\begin{split}
v(t_1,\ldots,t_k)&\leq\sum_{j=1}^k\big(\alpha_j\log|t_j|+(a_j-\alpha_j)\log^+|t_j|\big)+C\\
&\leq\sum_{j=1}^k\big((\nu+a_j-a)\log|t_j|+(a-\nu)\log^+|t_j|\big)+C\,,
\end{split}\]
which is the desired conclusion.

\medskip

{\em Step 3.} We complete the proof of the theorem, for the case of arbitrary dimensions $n_j\geq1$. This follows immediately from Steps 1 and 2 by using $k$-dimensional slices of ${\mathbb C}^n$ as we now indicate. Fix $\zeta^j\neq0$, $1\leq j\leq k$, and consider the function 
\[v(t_1,\ldots,t_k)=u\left(t_1\frac{\zeta^1}{|\zeta^1|},\ldots,t_k\frac{\zeta^k}{|\zeta^k|}\right)\,,\,\;(t_1,\ldots,t_k)\in{\mathbb C}^k\,.\]
Since $u$ satisfies \eqref{e:Lp}, we have that $v$ is psh on ${\mathbb C}^k$ and verifies \eqref{e:L1} with the constant $C$ from \eqref{e:Lp}. Moreover $\nu(u,0)\leq\nu(v,0)$. The conclusions of the theorem now follow from the ${\mathbb C}^k$ case, since $u(\zeta^1,\ldots,\zeta^k)=v(|\zeta^1|,\ldots,|\zeta^k|)$. 
\end{proof}

\section{Proofs of the main results}\label{S:proofs}

In this section we give the proofs of Theorems \ref{T:t1}--\ref{T:t4}. 

\begin{proof}[Proof of Theorem \ref{T:t1}]
Let $x\in X$. We may assume that $x=0\in{\mathbb C}^n$. We have $T\mid_{{\mathbb C}^n}=dd^cu$ for a psh function $u$ satisfying \eqref{e:Lp}, so by Proposition \ref{P:psh}, $\nu(T,x)=\nu(u,0)\leq a$. 

Assume now that $E_a(T)\neq\emptyset$. By \eqref{e:omp}, $T=\omega+dd^c\varphi$ for an $\omega$-psh function $\varphi$ on $X$. For $1\leq j\leq k$ we define 
\[L_j=\Span\pi_j(E_a(T))\subset{\mathbb P}^{n_j}\,,\,\;\ell_j=\dim L_j\,.\]
Let $\eta_j:{\mathbb C}^{n_j+1}\to{\mathbb C}^{n_j-\ell_j}$ be a surjective linear map such that $L_j=\Pi_j(\ker\eta_j\setminus\{0\})$.

\medskip

For $1\leq j\leq k$ we will prove that $\ell_j\leq n_j-1$ and 
\begin{equation}\label{e:ind}
\begin{split}
&\exists\,\psi_j \text{ $\omega$-psh on } Y_j:={\mathbb P}^{n_1-\ell_1-1}\times\ldots\times{\mathbb P}^{n_j-\ell_j-1}\times{\mathbb P}^{n_{j+1}}\times\ldots\times{\mathbb P}^{n_k} \text{ such that, on $X$,}\\
&\varphi([z^1],\ldots,[z^k])=\sum_{m=1}^ja_m\log\frac{|\eta_m(z^m)|}{|z^m|}+\psi_j([\eta_1(z^1)],\ldots,[\eta_j(z^j)],[z^{j+1}],\ldots,[z^k]),
\end{split}
\end{equation}
where, by abuse of notation, $\omega=\sum_{m=1}^ka_m\pi^\star_m\omega_{FS}$ and $\pi_m$ is the projection of $Y_j$ onto its $m$-th factor. Moreover, $[\eta_m(z^m)]=[\eta_m]([z^m])$, where $[\eta_m]:{\mathbb P}^{n_m}\dashrightarrow{\mathbb P}^{n_m-\ell_m-1}$ is the projection induced by $\eta_m$. The proof is by induction on $j=1,\ldots,k$. 

\medskip

Let $j=1$ and fix $x'\in{\mathbb P}^{n_2}\times\ldots\times{\mathbb P}^{n_k}$ such that $\varphi(\cdot,x')\not\equiv-\infty$. We claim that $\pi_1(E_a(T))\subset E_{a_1}(R)$, where $R$ is the positive closed current on ${\mathbb P}^{n_1}$ defined by $R=a_1\omega_{FS}+dd^c\varphi(\cdot,x')$. Indeed, let $p\in\pi_1(E_a(T))$ and fix $q\in{\mathbb P}^{n_2}\times\ldots\times{\mathbb P}^{n_k}$ such that $(p,q)\in E_a(T)$. Without loss of generality we may assume that $(p,q)=0\in{\mathbb C}^n$ and $x'\in{\mathbb C}^{n_2}\times\ldots\times{\mathbb C}^{n_k}$. Proposition \ref{P:psh} applied to the psh function 
\begin{equation}\label{e:uphi}
u(\zeta^1,\ldots,\zeta^k)=\sum_{m=1}^ka_m\log\sqrt{1+|\zeta^m|^2}\,+\varphi([1:\zeta^1],\ldots,[1:\zeta^k])
\end{equation}
shows that 
\[u(\zeta^1,\ldots,\zeta^k)\leq\sum_{m=1}^ka_m\log|\zeta^m|+C \;\text{ on ${\mathbb C}^n$}\,,\]
for some constant $C$. Since $u(\cdot,x')\not\equiv-\infty$ this implies that $\nu(u(\cdot,x'),0)\geq a_1$. As $u(\cdot,x')\in a_1{\mathcal L}(\mathbb C)$ it follows by Proposition \ref{P:proj} that $\nu(u(\cdot,x'),0)=a_1$. Hence $p\in E_{a_1}(R)$, which proves our claim. 

By Proposition \ref{P:proj}, $E_{a_1}(R)$ is a proper linear subspace of ${\mathbb P}^{n_1}$, so $L_1\subset E_{a_1}(R)$ and $\ell_1\leq n_1-1$. Moreover
\begin{equation}\label{e:d1}
\varphi([z^1],x')=a_1\log\frac{|\eta_1(z^1)|}{|z^1|}+\psi_1([\eta_1(z^1)],x')\,,
\end{equation}
where $\psi_1(\cdot,x')$ is $a_1\omega_{FS}$-psh on ${\mathbb P}^{n_1-\ell_1-1}$. Define $\psi_1(\cdot,x')\equiv-\infty$ for $x'$ such that $\varphi(\cdot,x')\equiv-\infty$. We conclude that \eqref{e:d1} holds on $X$ and $\psi_1$ is $\omega$-psh on $Y_1$. Hence \eqref{e:ind} holds for $j=1$.

\medskip

We assume next that \eqref{e:ind} holds for $j-1<k$ and prove it for $j$. Then 
\begin{equation}\label{e:d2}
\varphi([z^1],\ldots,[z^k])=\sum_{m=1}^{j-1}a_m\log\frac{|\eta_m(z^m)|}{|z^m|}+\psi_{j-1}[\eta_1(z^1)],\ldots,[\eta_{j-1}(z^{j-1})],[z^j],\ldots,[z^k]),
\end{equation}
holds on $X$, where $\psi_{j-1}$ is $\omega$-psh on $Y_{j-1}$. Fix $y=(y_1,\ldots,y_{j-1})\in{\mathbb P}^{n_1-\ell_1-1}\times\ldots\times{\mathbb P}^{n_{j-1}-\ell_{j-1}-1}$ and $x'\in{\mathbb P}^{n_{j+1}}\times\ldots\times{\mathbb P}^{n_k} $ such that $\psi_{j-1}(y,\cdot,x')\not\equiv-\infty$. Note that if $x''=(x_1,\ldots,x_{j-1})$ is such that $x_m\in{\mathbb P}^{n_m}\setminus L_m$ and $[\eta_m](x_m)=y_m$, $1\leq m\leq j-1$, then by \eqref{e:d2}, $\varphi(x'',\cdot,x')\not\equiv-\infty$. We claim that $\pi_j(E_a(T))\subset E_{a_j}(R)$, where $R$ is the positive closed current on ${\mathbb P}^{n_j}$ defined by $R=a_j\omega_{FS}+dd^c\psi_{j-1}(y,\cdot,x')$. Indeed, let $p\in\pi_j(E_a(T))$ and fix $q''\in{\mathbb P}^{n_1}\times\ldots\times{\mathbb P}^{n_{j-1}}$, $q'\in{\mathbb P}^{n_{j+1}}\times\ldots\times{\mathbb P}^{n_k}$ such that $(q'',p,q')\in E_a(T)$. Without loss of generality we may assume that $(q'',p,q')=0\in{\mathbb C}^n$ and $x''\in{\mathbb C}^{n_1}\times\ldots\times{\mathbb C}^{n_{j-1}}$, $x'\in{\mathbb C}^{n_{j+1}}\times\ldots\times{\mathbb C}^{n_k}$. Applying Propositions \ref{P:psh} and \ref{P:proj} to the psh function $u$ from \eqref{e:uphi}, we infer that $\nu(u(x'',\cdot,x'),0)=a_j$. By \eqref{e:d2} and the choice of $x''$ this implies that $\nu(\psi_{j-1}(y,\cdot,x'),0)=a_j$. So $p\in E_{a_j}(R)$, and our claim is proved. 

By Proposition \ref{P:proj}, $E_{a_j}(R)$ is a proper linear subspace of ${\mathbb P}^{n_j}$, so $L_j\subset E_{a_j}(R)$ and $\ell_j\leq n_j-1$. Moreover
\begin{equation}\label{e:d3}
\psi_{j-1}(y,[z^j],x')=a_j\log\frac{|\eta_j(z^j)|}{|z^j|}+\psi_j(y,[\eta_j(z^j)],x')\,,
\end{equation}
where $\psi_j(y,\cdot,x')$ is $a_j\omega_{FS}$-psh on ${\mathbb P}^{n_j-\ell_j-1}$. Define $\psi_j(y,\cdot,x')\equiv-\infty$ for $y,x'$ such that $\psi_{j-1}(y,\cdot,x')\equiv-\infty$. By \eqref{e:d2} and \eqref{e:d3} we conclude that 
\[\varphi([z^1],\ldots,[z^k])=\sum_{m=1}^ja_m\log\frac{|\eta_m(z^m)|}{|z^m|}+\psi_j([\eta_1(z^1)],\ldots,[\eta_j(z^j)],[z^{j+1}],\ldots,[z^k])\]
holds on $X$, hence $\psi_j$ is $\omega$-psh on $Y_j$. This concludes the proof of \eqref{e:ind} by induction on $j$. 

\medskip

By the definition of $L_j$, we have $\pi_j(E_a(T))\subset L_j$ for $1\leq j\leq k$, so $E_a(T)\subset L_1\times\ldots\times L_k$. On the other hand, formula \eqref{e:ind} for $j=k$ shows that $E_a(T)\supset L_1\times\ldots\times L_k$. Consider the current $S=\omega+dd^c\psi_k$ on $Y:=Y_k$. Then $S\in\mathcal T_{a_1,\ldots,a_k}(Y)$ and by \eqref{e:ind}, $T=\wp^\star S$, where $\wp=[\eta_1]\times\ldots\times[\eta_k]:X\dashrightarrow Y$. If $y\in E_a(S)$ and $x=(x_1,\ldots,x_k)$ is such that $x_j\in{\mathbb P}^{n_j}\setminus L_j$ for $1\leq j\leq k$, and $\wp(x)=y$, then $\nu(T,x)\geq\nu(S,y)$, so $x\in E_a(T)$, a contradiction. Thus $E_a(S)=\emptyset$ and the proof is of Theorem \ref{T:t1} is complete. 
\end{proof}

\smallskip

\begin{proof}[Proof of Theorem \ref{T:t2}]
$(i)$ We write $T=\omega+dd^c\varphi$, where $\varphi$ is an $\omega$-psh function on $X$, and we assume without loss of generality that $j=1$. Set $X'={\mathbb P}^{n_2}\times\ldots\times{\mathbb P}^{n_k}$ and $E=\{x'\in X':\,\varphi(\cdot,x')\equiv-\infty\}$. Note that $E$ is locally pluripolar, since $E\subset\{x'\in X':\,\varphi(x_1,x')=-\infty\}$ for some fixed $x_1\in{\mathbb P}^{n_1}$ such that $\varphi(x_1,\cdot)\not\equiv-\infty$. 

For $x'\in X'\setminus E$ define $T_{x'}=a_1\omega_{FS}+dd^c\varphi(\cdot,x')$. Then $T_{x'}$ is a positive closed current of bidegree $(1,1)$ on ${\mathbb P}^{n_1}$ of mass $\|T_{x'}\|=a_1$. We claim that 
\[\pi_1\big(E^+_{\nu_1}(T)\big)\subset E^+_{\nu_1+a_1-a}(T_{x'})\,.\] 
Indeed, let $p\in\pi_1\big(E^+_{\nu_1}(T)\big)$ and fix $q\in X'$ such that $(p,q)\in E^+_{\nu_1}(T)$. Without loss of generality we may assume that $(p,q)=0\in{\mathbb C}^n$ and $x'\in{\mathbb C}^{n_2}\times\ldots\times{\mathbb C}^{n_k}$. Note that 
\[\nu:=\nu(T,(p,q))>\nu_1\geq a-a_1\,,\] 
so Theorem \ref{T:psh} applied to the psh function $u$ defined in \eqref{e:uphi} shows that 
\[u(\zeta^1,\ldots,\zeta^k)\leq(\nu+a_1-a)\log|\zeta^1|+ (a-\nu)\log^+|\zeta_1|+\sum_{\ell=2}^ka_\ell\log^+|\zeta^\ell|+C \;\text{ on ${\mathbb C}^n$}\,,\]
for some constant $C$. Since $u(\cdot,x')\not\equiv-\infty$ this implies that 
\[\nu(u(\cdot,x'),0)\geq\nu+a_1-a>\nu_1+a_1-a=\frac{n_1a_1}{n_1+1}\;.\]
Thus $p\in E^+_{\nu_1+a_1-a}(T_{x'})$, and our claim is proved.

Note that $E^+_{\nu_1+a_1-a}(T_{x'})=E^+_{n_1/(n_1+1)}(T_{x'}/a_1)$ and that the current $T_{x'}/a_1$ has unit mass. We infer by \cite[Proposition 2.2]{CG09} that $\Span E^+_{\nu_1+a_1-a}(T_{x'})$ is a proper linear subspace of ${\mathbb P}^{n_1}$. It follows that 
\[\pi_1\big(E^+_{\nu_1}(T)\big)\subset V_1:=\bigcap_{x'\in X'\setminus E}\Span E^+_{\nu_1+a_1-a}(T_{x'})\,.\]

\smallskip

$(ii)$ This follows from $(i)$, since $\pi_j\big(E^+_{\nu_0}(T)\big)\subset\pi_j\big(E^+_{\nu_j}(T)\big)\subset V_j$ for all $1\leq j\leq k$.  
\end{proof}

\smallskip

\begin{proof}[Proof of Theorem \ref{T:t3}]
$(i)$ We assume without loss of generality that $j=1$ and use the same notation as in the proof of Theorem \ref{T:t2}. Fix $x'\in X'\setminus E$. Since $\nu_1\geq a-a_1$, it follows as in the proof of Theorem \ref{T:t2} that $\pi_1\big(E^+_{\nu_1}(T)\big)\subset E^+_{\nu_1+a_1-a}(T_{x'})$. Note that 
\[\nu_1+a_1-a=\frac{(n_1-1)a_1}{n_1}\;.\] 
We infer by \cite[Theorem 1.2]{CT15} that the set $E^+_{\nu_1+a_1-a}(T_{x'})$ is either contained in a hyperplane of ${\mathbb P}^{n_1}$, or else it is a finite set and $|E^+_{\nu_1+a_1-a}(T_{x'})\setminus L|=n_1-1$ for some line $L$. 

If $\pi_1\big(E^+_{\nu_1}(T)\big)$ is not contained in a hyperplane, then neither is $E^+_{\nu_1+a_1-a}(T_{x'})$. It follows that $\pi_1\big(E^+_{\nu_1}(T)\big)$ is a finite set and $\pi_1\big(E^+_{\nu_1}(T)\big)\setminus L_1=A_1$ for some line $L_1$ and set $A_1$ with $|A_1|\leq n_1-1$. However since $\pi_1\big(E^+_{\nu_1}(T)\big)$ is not contained in a hyperplane, we must have that $|A_1|=n_1-1$ and $\Span(L_1\cup A_1)={\mathbb P}^{n_1}$.

\smallskip

$(ii)$ This follows readily from $(i)$.
\end{proof}

\smallskip

\begin{proof}[Proof of Theorem \ref{T:t4}]
We recall \cite[Theorem 1.1]{He19} for bidegree $(1,1)$ currents: Let $S$ be a positive closed current of bidegree $(1,1)$ on ${\mathbb P}^N$, $N\geq2$, with mass $\|S\|=b$, and let 
\begin{equation}\label{e:gamma}
\alpha>\frac{b(N-1)}{N}\;,\,\;\gamma=\gamma(N,b,\alpha)=\frac{bN(N-1)-\alpha}{N^2-1}\;.
\end{equation}
If $\nu(S,p)\geq\alpha$, $\nu(S,q)\geq\alpha$, for some points $p\neq q$, then either $E^+_\gamma(S)$ is contained in a hyperplane or there exists a complex line $L$ such that $|E^+_\gamma(S)\setminus L|=N-1$.

$(i)$  Following what was done in the proof of Theorem \ref{T:t3}, we let $j=1$, so $p_1 \neq q_1$. We have $\alpha_1>\nu_1\geq a-a_1$ and, since $\alpha_1\leq a$, we get $\beta_1\geq a-a_1$. We infer that $p_1,q_1\in E_{\alpha_1+a_1-a}(T_{x'})$ and $\pi_1\big(E^+_{\beta_1}(T)\big)\subset E^+_{\beta_1+a_1-a}(T_{x'})$. Note that 
\[\alpha_1+a_1-a>\nu_1+a_1-a=\frac{a_1(n_1-1)}{n_1}\;,\,\;\beta_1+a_1-a=\gamma(n_1,a_1,\alpha_1+a_1-a)\,,\]
where $\gamma$ is defined in \eqref{e:gamma}. Applying \cite[Theorem 1.1]{He19} gives us that if $E^+_{\beta_1 +a_1 - a}(T_{x'})$ is not contained in a hyperplane, then $|E^+_{\beta_1+a_1-a}(T_{x'})\setminus L|=n_1-1$ for some line $L$.  Thus $E^+_{\beta_1}(T)$ satisfies the conclusion.

Assertion $(ii)$ follows immediately from $(i)$.
\end{proof}

\smallskip

We conclude this section with a series of examples which show that our theorems are sharp. 

\begin{Example} 
We show here that the values $\nu_j$ and $\nu_0$ from Theorem \ref{T:t2} are sharp for the geometric properties of the corresponding upper level sets of Lelong numbers. Assume without loss of generality that $j=1$ and let $S=\{s_0,\ldots,s_{n_1}\}\subset{\mathbb P}^{n_1}$ such that $\Span S={\mathbb P}^{n_1}$. Consider the hyperplanes $H^m_1=\Span(S\setminus\{s_m\})\subset{\mathbb P}^{n_1}$, $0\leq m\leq n_1$, and fix some hyperplanes $H_j\subset{\mathbb P}^{n_j}$, $2\leq j\leq k$. We denote by $[H]$ the current of integration along a hyperplane $H\subset{\mathbb P}^N$ and we let 
\[T=\frac{a_1}{n_1+1}\,\sum_{m=0}^{n_1}\pi_1^\star[H^m_1]+\sum_{j=2}^ka_j\pi_j^\star[H_j]\,.\]
Then $T\in\mathcal T$, $E^+_{\nu_1}(T)=\emptyset$, and $E_{\nu_1}(T)=S\times H_2\times\ldots\times H_k$. So $\pi_1\big(E_{\nu_1}(T)\big)$ is not contained in any hyperplane of ${\mathbb P}^{n_1}$.

Assume without loss of generality that $\nu_0=\nu_1$. Then the above current $T$ shows that the value $\nu_0$ is sharp with respect to the geometric property from Theorem \ref{T:t2} $(ii)$.
\end{Example}

\begin{Example}
To show that the values $\nu_j$ and $\nu_0$ from Theorem \ref{T:t3} are sharp we use a similar idea as in the previous example. By \cite[Examples 3.5, 3.6]{C06} and \cite[Section 2.3]{CT15}, there exists, for every $N\geq2$, a positive closed current $R_N$ on ${\mathbb P}^N$ of bidegree $(1,1)$ and unit mass $\|R_N\|=1$, such that the set $E_{(N-1)/N}(R_N)$ is not contained in any hyperplane of  ${\mathbb P}^N$ and $|E_{(N-1)/N}(R_N)\setminus L|\geq N$ for every line $L\subset{\mathbb P}^N$. Assume without loss of generality that $j=1$ and fix some hyperplanes $H_j\subset{\mathbb P}^{n_j}$, $2\leq j\leq k$. We let 
\[T=a_1\pi_1^\star R_{n_1}+\sum_{j=2}^ka_j\pi_j^\star[H_j]\,.\]
Then $T\in\mathcal T$ and $E_{\nu_1}(T)\supset E_{(n_1-1)/n_1}(R_{n_1})\times H_2\times\ldots\times H_k$. So $\pi_1\big(E_{\nu_1}(T)\big)$ is not contained in any hyperplane of ${\mathbb P}^{n_1}$ and $|\pi_1\big(E_{\nu_1}(T)\big)\setminus L|\geq n_1$ for every line $L\subset{\mathbb P}^{n_1}$. 
\end{Example}

\begin{Example}
In Theorem \ref{T:t4}, we need to have two components $p_j\neq q_j$, otherwise the conclusion fails to hold.  To see this, we use the current from the first example  in \cite[Section 3]{He19} in the bidegree $(1,1)$ case. In particular we consider the hyperplanes $H^m_1$, $1\leq m \leq n_1$, from Example 4.1, and notice $s_0\in H^m_1$ for all $m$.  We consider the current $R_1$ on $\mathbb{P}^{n_1}$ given by
\[R_1= \frac{1}{n_1}\,\sum_{m=1}^{n_1} [H^m_1]\,.\] 
Observe that $\nu(R_1, s_0) = 1$, for any point $x\in L^m_1 =\Span\{s_0, s_m\}$, $\nu(R_1,x) = \frac{n_1 -1}{n_1}$, and $\bigcup_{m=1}^{n_1} L^m_1$ is not contained in a hyperplane.  Letting $[H_j]$ be as in the previous two examples, we consider again the current   
\[T=a_1\pi_1^\star R_1+\sum_{j=2}^ka_j\pi_j^\star[H_j]\,.\]
Note that for $\nu_1<\alpha_1\leq a$, $E_{\alpha_1}(T)=E_a(T)=\{s_0\}\times H_2 \times\ldots \times H_k$, so $\pi_1\big(E_{a}(T)\big)=\{s_0\}$. Moreover $\bigcup_{m=1}^{n_1} L^m_1\subset\pi_1\big(E_{\beta_1}^+(T)\big)$, so assertion $(i)$ of Theorem \ref{T:t4} does not hold for $T$.
\end{Example}

\section{Currents on projective spaces}\label{S:fr}
If $T$ is a positive closed current of bidegree $(1,1)$ on ${\mathbb P}^n$ with mass $\|T\|=1$, it is shown in \cite[Proposition 2.2]{CG09} that $E^+_{n/(n+1)}(T)$ is contained in a hyperplane of ${\mathbb P}^n$ (see also \cite[Theorem 1.1]{CT15} for the case of currents of any bidegree $(q,q)$). The value $n/(n+1)$ is sharp, as shown by the following example. Let $S=\{p_0,\ldots,p_n\}\subset{\mathbb P}^n$ such that $\Span S={\mathbb P}^n$, and consider the hyperplanes $H_j=\Span(S\setminus\{p_j\})$, $0\leq j\leq n$. If
\begin{equation}\label{e:pe}
T=\frac{1}{n+1}\,\sum_{j=0}^n\,[H_j]\,,
\end{equation}
then $E_{n/(n+1)}(T)=S$ is not contained in any hyperplane. We show here that such currents $T$ provide the only examples:

\begin{Theorem}\label{T:proj}
If $T$ is a positive closed current of bidegree $(1,1)$ on ${\mathbb P}^n$ and $\|T\|=1$, then either $E_{n/(n+1)}(T)$ is contained in a hyperplane of ${\mathbb P}^n$ or else $T$ is a current of form \eqref{e:pe}.
\end{Theorem}

We start with the following lemma:

\begin{Lemma}\label{L:proj}
Let $T$ be a positive closed current of bidegree $(1,1)$ on ${\mathbb P}^n$ such that $\|T\|=1$. Assume that, for some $\nu>0$, there exist linearly independent points $p_1,\ldots,p_n\in E_\nu(T)$. Then $T\geq(n\nu-n+1)[H]$, where $H$ is the hyperplane spanned by $p_1,\ldots,p_n$. 
\end{Lemma}

\begin{proof} The lemma is obviously true for $n=1$, so we assume $n\geq2$. We write $T=\alpha[H]+(1-\alpha)R$, where $\alpha$ is the generic Lelong number of $T$ along $H$ and $R$ is a positive closed current of bidegree $(1,1)$ on ${\mathbb P}^n$ such that $\|R\|=1$. Using Demailly's regularization theorem \cite{D92} we infer that there exists, for each $\varepsilon>0$, a positive closed current $R_\varepsilon=\omega_{FS}+dd^c\varphi_\varepsilon$ on ${\mathbb P}^n$ such that $\nu(R_\varepsilon,p_j)>\nu(R,p_j)-\varepsilon$, $1\leq j\leq n$, and $R_\varepsilon$ is smooth near each point $p$ where $\nu(R,p)=0$, hence near the generic point of $H$. Therefore 
\[S:=R_\varepsilon\mid_H=\omega_{FS}+dd^c(\varphi_\varepsilon\mid_H)\]
is a well defined positive closed current on $H\equiv{\mathbb P}^{n-1}$ and 
\[\nu(S,p_j)\geq\nu(R_\varepsilon,p_j)>\nu(R,p_j)-\varepsilon\geq\frac{\nu-\alpha}{1-\alpha}-\varepsilon\,,\,\;1\leq j\leq n\,.\]

By \cite[Proposition 2.2]{CG09} we have that $E^+_{(n-1)/n}(S)$ is contained in a hyperplane of ${\mathbb P}^{n-1}$. Since the points $p_1,\ldots,p_n$ are in general position, it follows that $p_j\not\in E^+_{(n-1)/n}(S)$ for some $j$. Hence 
\[\frac{n-1}{n}\geq\nu(S,p_j)>\frac{\nu-\alpha}{1-\alpha}-\varepsilon\;.\]
Letting $\varepsilon\searrow0$ this implies that $\alpha\geq n\nu-n+1$, so $T=\alpha[H]+(1-\alpha)R\geq(n\nu-n+1)[H]$. 
\end{proof}

\smallskip

\begin{proof}[Proof of Theorem \ref{T:proj}] Assume that $E_{n/(n+1)}(T)$ is not contained in any hyperplane of ${\mathbb P}^n$. Then there exists a set $S=\{p_0,\ldots,p_n\}\subset E_{n/(n+1)}(T)$ such that $\Span S={\mathbb P}^n$. If $H_j=\Span(S\setminus\{p_j\})$, then by Lemma \ref{L:proj}, 
\[T\geq\left(n\,\frac{n}{n+1}-n+1\right)[H_j]=\frac{1}{n+1}\,[H_j]\,,\,\;0\leq j\leq n\,.\]
Hence by Siu's decomposition theorem \cite{Siu74},
\[T\geq T':=\frac{1}{n+1}\,\sum_{j=0}^n\,[H_j]\,.\]
Since both currents $T,T'$ have unit mass it follows that $T=T'$.
\end{proof}

\end{document}